\documentclass[a4paper,10pt]{article}
 \usepackage{graphics,graphicx}
 \usepackage{amsmath,amsthm}
 \usepackage{amsfonts,mathrsfs,amssymb}

 \newtheorem{theorem}{Theorem}[section]
 \newtheorem{lemma}[theorem]{Lemma}
 \newtheorem{proposition}[theorem]{Proposition}
 \newtheorem{definition}[theorem]{Definition}

 \theoremstyle{definition}
 \newtheorem{remark}[theorem]{Remark}
 \newtheorem{example}[theorem]{Example}

  \newcommand{\mb}[1]{\ensuremath{\mathbb{#1}}}
  \newcommand{\N}{\mb{N}}
  
  \newcommand{\R}{\mb{R}}
  \newcommand{\C}{\mb{C}}
  
  \newcommand{\LT}[1]{\ensuremath{{\cal L}#1}}
  \newcommand{\ILT}[1]{\ensuremath{{\cal L}^{-1}#1}}

  \newcommand{\al}{\alpha}
  \newcommand{\be}{\beta}
  \newcommand{\ga}{\gamma}

  \newcommand{\vphi}{\varphi}
  
  \newcommand{\eps}{\varepsilon}

  \newcommand{\spp}{\ensuremath{{\cal S}}'_+}
  
  \newcommand{\E}{\ensuremath{{\cal E}}}
  \newcommand{\supp}{\mathop{\mathrm{supp}}}
  \renewcommand{\Re}{\ensuremath{\text{Re }}}

  \newcommand{\dis}[2]{\langle #1 , #2 \rangle}

  \title{Distributional framework for solving fractional differential equations}

  \author{Teodor M. Atanackovic\thanks{Faculty of Technical Sciences, University of Novi Sad,
  Trg D. Obradovi\'ca 5, 21000 Novi Sad, Serbia,
  (\tt{atanackovic@uns.ns.ac.yu})}\hspace{0,2cm}
  Ljubica Oparnica
  \thanks{Institute of Mathematics, Serbian Academy of Science,
  Kneza Mihaila 35, 11000 Belgrade, Serbia,
  (\tt{ljubicans@sbb.co.yu})}\hspace{0,2cm}
  Stevan Pilipovi\'{c}
  \thanks{Department of Mathematics and Informatics, University of Novi Sad,
  Trg Dositeja Obradovi\'ca 4, 21000 Novi Sad, Serbia,
  (\tt {pilipovic@im.ns.ac.yu})}
  }

\begin{document}
 \date{}
 \maketitle

 \begin{abstract}
  We
  analyze solvability of a special form of
  distributed order fractional
  differential equations
   \begin{equation}\label{ConstEq}
   \int_{0}^{2}\phi_{1}(\gamma)D^{\gamma}y(t)d\gamma =
   \int_{0}^{2}\phi _{2}(\gamma )D^{\gamma}z(t)d\gamma,\quad t>0,
   \end{equation}
 within  $\spp$, the space of tempered distributions
  supported by $[0,\infty)$.


 \noindent {\footnotesize AMS Subject 26A33; 46F12}
 \end{abstract}

 \section{Motivation and Introduction}\label{intro}

 We consider a distributed order fractional differential equation
 (\ref{ConstEq})
 which arises in the theory of constitutive equations for viscoelastic
 bodies. $\phi _{1},\phi _{2}$ are certain functions or distributions
 which characterize a material under consideration and, in general,
 are determined from experiments.
 $D^{\gamma }$, $\gamma \in \mathbb{R}$,
 is the left Riemann-Liouville operator of fractional differentiation or
 integration defined as follows.

 Denote by $L_{loc+}^{1}(\mathbb{R})$
 the space of locally integrable functions $y$ on $\R$
 such that
 $y(t)=0$, $t<0$.
 Then for $y\in L_{loc+}^{1}(\mathbb{R})$ the left fractional
 integral of order $\ga > 0$ is defined by
 \begin{equation*}
 I^{\ga}y(t):=\frac{1}{\Gamma (\ga)}\int_{0}^{t}(t-\tau )^{\ga-1}y(\tau )d\tau,\quad t >
 0.
 \end{equation*}
 Here $\Gamma $ is the Euler gamma function.
 If $\ga=0$ then $I^{0}y : = y$.
 It can be shown (cf. {\cite{SKM:93}})
 that for $y\in L_{loc}^{1}(\R)$
 $\lim_{\ga\to 0}I^{\ga}y(t)=y(t)$, $t\in\R$ almost everywhere.

 Denote by $AC^k(\overline{\R}_+)$ the space of
 functions $y$ such that $y$ has continuous derivatives on
 $\overline{\R}_+ = \{x\in\R;x \geq 0\}$
 up to the order $k-1$ and $k$-th
 derivative is locally integrable function. We extend such functions
 to $\R$ so that $y(t)=0$, $t<0$.

 Let $y\in AC^k(\overline{\R}_+)$. Riemann-Liouville's fractional derivative of order
 $\ga\geq 0$, $\ga\leq k$ for some $k\in\N$,
 is defined by
 \begin{equation*}
  D^{\ga}y(t) := \frac{d^{k}}{dt^{k}}I^{k-\ga}y(t),\quad t > 0.
 \end{equation*}
 It follows that $D^{\ga}y\in L^1_{loc+}(\R)$. We refer to Section \ref{preliminaries}
 for the definition of $D^{\ga}y$, $y\in \spp$ and
 $\ga\in\R$.

 Note that
 $D^{\ga}I^{\ga}y = y$ for $y\in L^1_{loc+}(\R)$
 and $I^{\ga}D^{\ga}y = y$, $\ga > 0$ in the sense of tempered
 distribution. We sometimes denote $D^{-\ga}y = I^{\ga}y$, $\ga > 0$.

 Let $\phi$ be  continuous function
 in $[c,d]\subset [0,k]$, $k\in\N$.
 Distributed order fractional derivative  of
 $y\in AC^k(\overline{\R}_+)$ is given by
 $$
 \int_c^d \phi(\ga)D^{\ga}y(t)d\ga.
 $$

 Equation (\ref{ConstEq})
 models various physical processes. For example,
 if it models a viscoelastic body then (\ref{ConstEq})
 represents  a constitutive equation of a material
 and connects  strain $ y(t)$
 with  corresponding stress $z(t)$
 at time instant $t\geq 0$.
 For a standard linear viscoelastic body the constitutive equation
 is given as
 \begin{equation*}\label{4a}
  y(t) + b y^{(1)}(t) =  z (t) + a z^{(1)}(t),
 \end{equation*}
 where $(\cdot)^{(1)}=\frac{d}{dt}(\cdot)$ and $a,b$
 are experimentally determined constants
 with the restriction $0<a<b$ following from the second law of
 thermodynamics.
 A slight generalization (see \cite{TMA:021} and references there) of this equation is achieved by
 replacing the first derivative
 by a derivative of real order $\al>0$
 \begin{equation}\label{zener}
   y(t) + bD^{\alpha}y(t) = z(t) + a D^{\alpha}z(t),
 \end{equation}
 where, again $0<a<b$.
 If $0<\al<1$ then (\ref{zener}) represents visco-elastic effects
 while for $1<\al<2$, (\ref{zener}) describes visco-inertial
 effects of a material.
 Standard procedure in building rheological models is to use more
 then one derivative on each side of constitutive equation. When
 this is done (\ref{zener}) becomes
 \begin{equation}
 \sum_{n=0}^{N}b_{n}D^{\be_{n}}y(t)
 =\sum_{m=0}^{M}a_{m}D ^{\al_{m}}z(t), \label{1}
 \end{equation}%
 where $M,N\in\N$, $a_{m},b_{n}\in\R$ and $\al_{m}, \be_{n}\in\R$,
 $0\leq \al_{m},\be_{n}\leq 2$.

 Equation (\ref{1})
 is interpreted in \cite{TMA:02} as a Riemann sum.
 Moreover, in \cite{TMA:02} is proposed the constitutive equation of
 a linear visco-elastic body in a \textquotedblright
 distributed\textquotedblright\ order form as
 \begin{equation}\label{ConstEq1}
  \int_{0}^{2}\phi _{1}(\gamma )D^{\gamma }y(t)d\gamma =
   \int_{0}^{2}\phi_{2}(\gamma )D^{\gamma }z(t)d\gamma,\quad t > 0.
 \end{equation}
 In model (\ref{ConstEq1}) all
 derivatives of the stress $D^{\ga}z$ depend on
 all derivatives of the strain $D^{\ga}y$ for $\ga\in [c,d]$.
 Since the upper bound in integrals
 in (\ref{ConstEq1}) is two, both, visco-elastic and visco-inertial effects
 are  included. The presence of integral on the left
 hand side indicates, as experiments show, that dissipation
 properties depend on the order of the derivative. The integral on the right
 hand side is a consequence of the known principle of
 equipresence.

 In this paper we are looking for an $\spp$ solution $z$
 to (\ref{ConstEq1}) for a given but arbitrary $y\in\spp$.
 Such solution will be  used  in \cite{ATLJOSP:08} for solving
 a differential equation of motion coupled with
 constitutive equation (\ref{ConstEq1}).

 In  Section \ref{preliminaries} we
 extend results obtained in \cite{SPLJOTA:06}
 concerning integral in (\ref{ConstEq1}).
 Afterwards we define distributed order fractional derivative in $\spp$ and
 derive it's main properties.
 In Section \ref{SecLFDE} we state without proof (which is given in \cite{ATLJOSP:08})
 a theorem on the existence and uniqueness of a solution  to a linear fractional
 differential equation in the frame of $\spp$.
 Also we derive properties of such solution. In Section
 \ref{MechRem1} we connect condition for the uniqueness
 with a dissipation inequality that guarantees physical
 admissability of a equation in (\ref{ConstEq1}).

 We note that the models with distributed order
 derivatives were analyzed for example  in  \cite{Atanackovic:03,Caputo:01,TASP:04,MPG:06,Nakhushev:03}.

 \section{Distributed order fractional derivative}\label{preliminaries}
 We denote by $\ensuremath{{\cal S}}(\mathbb{R})$ the space of rapidly decreasing
 functions in $\mathbb{R}$ and by $\ensuremath{{\cal S}}^{\prime }(\mathbb{R})$ its
 dual, the space of tempered distributions;
 $\ensuremath{{\cal S}}_{+}^{\prime }(\mathbb{R})$ denotes
 its subspace consisting of distributions supported by
 $[0,\infty )$. In the sequel we drop $\R$ in the notation.
 We consider in $\ensuremath{{\cal S}}_{+}^{\prime }$ the family
 \begin{equation}\label{falfa}
 f_{\alpha }(x)=\left\{
 \begin{array}{cc}
  H(x)\frac{x^{\alpha -1}}{\Gamma (\alpha )}, & x\in\R,\,\alpha > 0, \\
 \frac{d^{N}}{dx^{N}}f_{\alpha + N }(x), & \quad \alpha \leq 0,\,\alpha + N > 0 ,\, N \in
 \N,
 \
 \end{array}
 \right.
 \end{equation}
 where $H$ is Heaviside's function. It is known that
 $f_{\al}*f_{\be}=f_{\al+\be}, \al,\be\in\R$. The convolution operator $f_{\alpha }\ast $ in
 $\spp$ is the operator of fractional differentiation for $\alpha < 0$ and of
 fractional integration for $\alpha > 0$. It coincides with
 the operator of derivation for $-\al\in\N$ and integration for $\alpha \in \N$.
 Let $\alpha > 0$ and $y\in L_{\text{loc+}}^{1}(\R)$.
 Then $I^{\alpha }y = f_{\alpha }\ast y$.
 Let $y\in AC^k(\overline{\R}_+)$ and  $0 < \al\leq k$.
 Then $D^{\alpha}y = f_{-\alpha }\ast y$.

 Recall, if $y\in \ensuremath{{\cal S}}'_+$ then its Laplace
 transform is defined by
 \begin{equation*}
 \widehat{y}(s)=\mathcal{L}y(s)=\langle y(t),\varphi (t)e^{-st}\rangle ,\quad
 \Re {s}>0,\quad \al\in\R,
 \end{equation*}%
 where $\varphi \in C^{\infty }$, $\varphi =1$ on $(-a,\infty )$ and $\varphi =0$
 in $(-\infty ,-2a)$, $a>0$. Note that $\mathcal{L}y$ is an analytic function for
 $\Re s>0$ and that the definition of $\LT y$
 does not depend on a chosen function $\varphi $ with given
 proprieties. We will often use the identity
 $$\LT({f_{\alpha}\ast y})(s) = \frac{1}{s^{\alpha }}\hat{y}(s),\quad \Re s > 0.$$

 First we analyze integral
 $\underset{\supp \phi}{\int}\phi(\gamma)D^{\gamma}y(\cdot)d\gamma$. To do this
 we examine the mapping  $\alpha\mapsto D^{\alpha}y: \R\to\spp $,
 for given $y\in \spp$ (in \cite{SPLJOTA:06} we have considered
 $y\in L_{loc}^{1}(\R)\cap \spp$).

 \begin{proposition}\label{P1}
 (a) Let $\alpha\in\mathbb{R} $ be fixed. Then the mapping
 $y \mapsto D^{\alpha}y$ is  linear and  continuous from
 $\ensuremath{{\cal S}}^{\prime}_+$ to $\ensuremath{{\cal S}}^{\prime}_+$.\newline
 (b) Let $y\in\ensuremath{{\cal S}}'_+ $ be fixed. Then $\alpha\mapsto D^{\alpha}y$
 is a smooth mapping from $\mathbb{R} $ to $\ensuremath{{\cal S}}^{\prime}_+$.\newline
 (c) The mapping $(\alpha,y) \mapsto D^{\alpha}y$ is
 continuous from $\mathbb{R} \times\ensuremath{{\cal S}}^{\prime}_+$
 to $ \ensuremath{{\cal S}} ^{\prime}_+$.
 \end{proposition}
 \begin{proof}
  (a) The continuity of $y \mapsto D^{\al}y = f_{-\alpha} * y$ is clear since
  for $g\in\spp$, $f \mapsto f*g$ is a continuous mapping of $\spp$ into $\spp$. \\
  (b) It is known that there exists a continuous function $F$,
  $\supp F\subset [0,\infty)$ and $k\in\N$ such that $|F(x)| < C (1+|x|)^k$, $x\in\R$ and $y = D^kF $.
  So the mapping $\alpha\mapsto D^{\alpha}y$ equals
  $\alpha\mapsto D^{\alpha+k}F$.
  By \cite[Proposition 1]{SPLJOTA:06} we know that for
  fixed $k$ and $\al\in\R$,  $\alpha + k \mapsto D^{\alpha+k}F$ is smooth
  so the same hold for $\al\mapsto D^{\alpha+k}F $.\\
  (c) Since $\ensuremath{{\cal S}}$ is Fr\'echet space as well as locally convex,
  the separate continuity proved in  (a) and (b) imply joint continuity
  (c.f. \cite[Corollary to Theorem 34.1]{Treves:67})  .
 \end{proof}
 By $\E'(\R)$ is denoted the space of compactly supported distributions
 i.e. the dual space of $\E(\R) = C_0^{\infty}(\R)$.
 \begin{definition}
 Let $\phi\in\E'(\R)$ and $y\in\spp$.
 Then
 $ \underset{\supp \phi}{\int}\phi(\gamma)D^{\gamma}y\, d\gamma$
 is defined as an element of $\spp$ by
 \begin{equation}  \label{DefInt}
 \langle \underset{\supp \phi}{\int} \phi(\gamma)D^{\gamma}y(t) \,d\gamma , \varphi(t) \rangle=
 \langle \phi(\gamma) , \langle D^{\gamma}y(t) , \varphi(t) \rangle \rangle,
 \quad \varphi\in{\cal S} (\mathbb{R} ).
 \end{equation}
 Such defined distribution is called distributed order fractional
 derivative.
 \end{definition}
 By  Proposition \ref{P1}, part (b),
 $\gamma\mapsto D^{\gamma}y:\mathbb{R} \to\spp$  is smooth
 as well as $\gamma\mapsto\langle D^{\gamma}y(t) , \varphi(t) \rangle : \R\to\R $.
 Since $\ensuremath{{\cal S}}$ is a Fre\'chet space it follows that in it's
 dual space the strong and weak boundedness are the same, thus a linear functional
 defined by (\ref{DefInt}) is continuous from $\ensuremath{{\cal S}}$ to $\C$
 and therefore is a tempered distribution supported by
 $[0,\infty)$.

 The following two examples are often used in applications.
 \begin{example}\label{ex1}
 a) Let $\gamma_i\in\R$, $i\in\{0,1,..,k\}$ and
 $\phi(\cdot) = \sum_{i=0}^k a_i \delta^{(i)}(\cdot-\gamma_i)$.
 Then (\ref{DefInt}) gives
 \begin{equation*}
   \underset{\supp \phi}{\int}\phi(\gamma)D^{\gamma}y (\cdot)\, d\gamma = \sum_{i=0}^k a_i
   D^{\gamma_i}y(\cdot),\quad \text { in }\spp .
 \end{equation*}

 b) Let $\phi$ be a continuous function in $[c,d]\subset\R$ for some $c<d$, then
 \begin{equation*}
  \int_c^d \phi(\gamma)D^{\gamma}y(\cdot)\, d\gamma =
  \lim_{N\to\infty}\sum_{i=1}^{N}\phi(\gamma_i)D^{\gamma_i}y(\cdot)
  \Delta\gamma_i,  \quad\text{in $\spp$},
 \end{equation*}
 where $\gamma_{i}$ are points of interval $[c,d]$ in usual definition of
 the Riemann sum for the  integral.
 \end{example}

 \begin{proposition} \label{inlap}
 Let $\phi\in\E'(\R)$ and $y\in\spp$.  Then:\\[0.1cm]
 a) $$y \mapsto \underset{\supp \phi}{\int}\phi(\gamma)D^{\gamma}y \,d\gamma $$
  is a linear and continuous mapping from $\spp $ to  $\spp$.\\
 b) $${\mathcal{L}}(\underset{\supp \phi}{\int} \phi(\gamma)D^{\gamma}y\,d\gamma)(s)
   = \hat{y}(s)\langle \phi(\gamma) , s^{\gamma} \rangle, \quad Re s> 0.$$
 c) If $\phi $ is continuous function on $[c,d]$ and $\phi(\ga)=0 $
 for $\ga\notin [c,d]$ then
 $${\mathcal{L}}(\int_c^d \phi(\gamma)D^{\gamma}y\,d\gamma)(s)
   = \hat{y}(s)\int_c^d \phi(\gamma)s^{\gamma} \,d\gamma, \quad Re s> 0.$$
 \end{proposition}

 \begin{proof} a) Clearly, this mapping is linear. Let $y_n \to 0$ in $\spp$. Then
 $\dis{\phi(\ga)}{\dis{D^{\ga}y_n}{\vphi}} \to  0$, as $n\to\infty$,
  since by Proposition \ref{P1} part (a), $\dis{D^{\ga}y_n}{\vphi}\to 0$, as $n\to\infty$.\\[0.1cm]
  b) By the definition,
  \begin{align*}
  \LT(\underset{\supp \phi}{\int} &\phi(\gamma)D^{\gamma}y \,d\gamma)(s)
   = \langle \underset{\supp \phi}{\int} \phi(\gamma)D^{\gamma}y(t) \,d\gamma , \varphi(t) e^{-st}
  \rangle\\
  &=\langle \phi(\gamma) , \langle D^{\gamma}y\,(t) , \varphi(t) e^{-st} \rangle
  \rangle = \langle \phi(\gamma) , s^{\gamma}\hat{y}(s) \rangle,\quad \Re s> 0.
  \end{align*}
 c) In the case that $\phi$ is continuous we have \underline{}
 $$ \langle \phi(\gamma) , s^{\gamma}\hat{y}(s) \rangle
 =\underset{\supp\phi}{\int}\phi(\gamma)s^{\gamma}\hat{y}(s) \,d\gamma $$
 and therefore the assertion follows.
 \end{proof}

 If we assume that $y,z\in\spp$ in (\ref{ConstEq1}),
 put $\phi = \phi_2$ and $g=\int_{\supp\phi_1}\phi_1(\ga)D^{\ga} z$
 then the solvability of
 (\ref{ConstEq1}) with respect to $z$
 reduces to the solvability
 of
 \begin{equation}\label{DOFDE}
 \int_{\supp\phi}\phi(\ga)D^{\ga} z = g,\quad g\in\spp.
 \end{equation}

 \section{Linear fractional differential equation in $\spp$}\label{SecLFDE}
 Assuming that  $g\in\spp$  and that $\phi$ is of the form as in Example
 \ref{ex1}, equation (\ref{DOFDE}) becomes
 \begin{equation}  \label{LFDE}
   \sum_{i=0}^k a_i D^{\gamma_i}z = g,\quad {\text {in } }\,\spp.
 \end{equation}
 We suppose that $\ga_i\in [0,2)$ such that $\ga_0 > \ga_i > \ga_{i+1}> \ga_k$,
 $i\in\{1,...k-1\}.$
 \begin{theorem}\label{T1}
 Equation (\ref{LFDE}) has a unique solution $z\in\spp$ if and only if
 \begin{align*}
   (A_0)\quad\quad \sum_{i=0}^k a_i s^{\gamma_i}\neq 0,
   \quad s\in\mathbb{C} _+ = \{ s\in\C\,;\, \Re s >0\}.
  \end{align*}
 \end{theorem}
 The proof is given in \cite{ATLJOSP:08}.
 The solution to (\ref{LFDE}) that is obtain in Theorem \ref{T1} is given by
 $z=l*g$, where
 \begin{equation}\label{l-fund}
 l(t)=\ILT(\frac{1}{\sum_{i=0}^k a_i s^{\ga_i}})(t),\quad t>0,
 \end{equation}
 is a fundamental solution to (\ref{LFDE}) i.e. solution to
 $\sum_{i=0}^k a_i D^{\ga_i}y=\delta.$

 The following lemma gives main properties of $l$ defined by (\ref{l-fund}).
 \begin{lemma}\label{mg}
  Assume $(A_0)$. Let $\ga_i\in [0,2)$ and $\ga_0 > \ga_i > \ga_{i+1}> \ga_k$,
 for all $i\in\{1,..,k-1\}$. Let $l$ be defined by (\ref{l-fund})
 and $l(t)=0$, $t<0$.  Then:

  (i) $l$ is a locally integrable function in $\R$.

  (ii) Moreover, $l$ is absolutely continuous in $\R$,  if $\ga_0 - \ga_k > 1$.
 \end{lemma}

 \begin{proof}
 (i) Let $\ga_k = 0$ and $a_k\neq 0$. Consider the integral
  $$
 \int_{\Gamma}\frac{e^{st}ds}{\sum_{i=0}^k a_is^{\gamma_i}},\quad
 t>0,
  $$
 where $\Gamma=\bigcup_{i=1}^5 \Gamma_i$
  and for arbitrarily chosen $R>0$, $0<\eps<R$ and $x_0 > 0$,
  $\Gamma_i$ are given by
 $$
 \Gamma_0:\, \{z; \Re z=x_0;\,0< \arg z < \phi_0 = \arcsin {\frac{x_0}{R}}\};
 $$
 $$
 \Gamma_1:\, z = Re^{i\varphi},\,-\phi_0 < \phi_0 \leq \varphi< \pi;
 \quad \Gamma_2:\, z = Re^{i\varphi},\,-\pi < \varphi\leq -\phi_0<0;
 $$
 $$
 \Gamma_3:\, z=\eps e^{i\varphi},\, -\pi < \varphi<\pi;
 \quad
  \Gamma_4:\, z=xe^{i\pi};\quad \Gamma_5:\,z=xe^{-i\pi},\quad
  x\in(\eps,R).
 $$
 By the Cauchy residue theorem, letting $\eps\to 0$ and
 $R\to\infty$,
 one obtains
 \begin{equation}\label{res}
  l(t)= \sum_{s=s_m,\\ m=1}^n \text{Res} \{\frac{e^{st}}{\sum_{i=0}^k a_is^{\gamma_i}} \} + l_0(t),
 \quad t>0,
 \end{equation}
 where $s_m$ are poles of  the function $s\mapsto\frac{e^{st}}{\sum_{i=0}^k
 a_is^{\gamma_i}}$ and
 \begin{equation}\label{integ}
 l_0(t) = \frac{1}{\pi}\int_0^{\infty} e^{-st}r(s)ds,\quad t>0,
 \end{equation}
 with
 \begin{equation*}
 r(s)=\frac{2\sum_{i=0}^k a_is^{\gamma_i}\sin(\gamma_i\pi)} {
 \sum_{i=0}^k a_i^2s^{2\gamma_i} + 2\sum_{i,j=0,i\neq j}^k a_ia_j s^{\gamma_i +
 \gamma_j} \cos(\gamma_i-\gamma_j)\pi},\ s>0.
 \end{equation*}
  We refer to \cite{TALJOSP:06} for similar calculations.

 Let $0\leq a\leq b$. Then
 \begin{equation}\label{l00}
  \int_a^bl_0(t)dt= \int_0^\infty(e^{-sa}-e^{-sb})\frac{1}{s}r(s)ds.
 \end{equation}
  Since $\ga_k=0$, this integral is finite in a neighborhood of $s=0$. In
  a neighborhood of $ s = \infty $ we have
 $ \frac{r(s)}{s} \sim \frac{1}{s^{\ga_0 +1}}$.
 Thus
 $\ga_0+1 > 1$  implies that the integral in (\ref{l00}) is finite.
 Therefore, $l_0$ is locally integrable.
 By  (\ref{res}) and the fact that $\Re s_m < 0$ (by
 $(A_0)$),
 we obtain that $l$ is locally integrable.

  Let $\ga_k > 0$. Then
 \begin{equation}
  l =  \ILT( \frac{1}{s^{\gamma_k}})* \ILT(\frac{1}{\sum_{i=0}^k a_is^{\gamma_i-\gamma_k}})
    = f_{\ga_k}*l_1,
  \end{equation}
  where
  $$l_1=\ILT(\frac{1}{\sum_{i=0}^k a_is^{\be_i}}),
  \quad \be_i = \gamma_i-\gamma_k,\quad i\in \{0,1,..,k\},
  \quad \Re s > 0,
  $$
  and $f_{\ga_k}$ is defined by  (\ref{falfa}).
 Note that $\be_k = 0$. By the first part of the proof, $l_1$ is a locally integrable function.
 Since $f_{\ga_k}$ is locally integrable, $l$ is locally integrable as the convolution of two locally
 integrable functions.

 (ii) Let $\ga_k = 0$. Then (\ref{integ}) is finite in a neighborhood
 of $s=0$. In a neighborhood of $s=\infty$ we have that
 $r(s)\sim \frac{1}{s^{\ga_0}}$.
 Since $\ga_0 > 1$, the integrand in (\ref{integ}) is integrable
 for all $t>0$ and (\ref{integ}) is finite.
 Let $t_0>0$.
 Since  $|e^{-st}r(s)|\leq e^{-st_0}|r(s)|:=g(s)$ for all $t > t_0$ and $g\in
 L^1((0,\infty))$, by the classical theory
 we obtain that (\ref{integ}) defines
 a continuous function for $t>t_0$. It follows that $l_0$ and $l$
 (by (\ref{res})) are continuous for $t>0$.
 Further on, since
 $|\partial_t(e^{-st}r(s))|\leq e^{-st_0}|sr(s)|:=g_1(s)$,
 for all $t > t_0$
 and $g_1\in L^1(0,\infty)$, we obtain that $l_0$ is differentiable
 and
 $$l'_0(t) = \int_0^\infty (-s)e^{-st}r(s),\, t>0.$$
 Since $-sr(s)\sim \frac{1}{s^{\ga_0-1}}$ in a neighborhood of $s=\infty$, as in (i) we show
 that $l'_0$ is a locally integrable function. Therefore,
 the derivative $l'_0$ exists and it is a locally integrable function. It
 means that $l$ is absolutely continuous.

 For $\ga_k > 0$ we proceed as in (i) and obtain $l = f_{\ga_k}*l_1$ with
 $l_1$ absolutely continuous. Therefore,  $l$ is also absolutely continuous.
\end{proof}

 \begin{remark}
  If $\ga_i\in[0,\infty)$ and if $\ga_0 > p$, $p \in\N$, then $l$ is continuous in
  $\R$ as well as its derivatives up to order $p-1$ while the $p$-th derivative is a
  locally integrable function, i.e. $l\in AC^p$.
 \end{remark}

 \section{Comments from mechanics and further applications}\label{MechRem1}
  Equation (\ref{LFDE}) represents a constitutive equation of a
  visco-elastic body.
  We will show that in the case when  there exist
  $s_0\in \C_+$ such that
  $\sum_{i=0}^{k}a_{i}s_{0}^{\gamma _{i}}=0$ it follows that
  the dissipation inequality (\ref{dod1}), (see \cite{Christensen:82}) is violated.
  The dissipation
  inequality requires that for any $T>0$,
  any $y$ and $z$ the solution to  $\sum_{i=0}^{k}a_{i}D^{\ga_{i}}z(t)= y(t)$,
  $t>0$,
  the dissipation work, $A_d$ is nonnegative, i.e.
   \begin{equation} \label{dod1}
    A_{d}=\int_{0}^{T}z(t)y^{(1)}(t)dt \geq 0.
   \end{equation}
  Let $T>0$ and $y(t) = H(t)-H(t-\tau)$, $0 < \tau < T $, $t>0$.
  Then
 \begin{eqnarray}\label{dod2}
  z(t) = \ILT(\frac{1}{s}\frac{1}{\sum_{i=0}^{k}a_{i}s^{\ga_i}}
     - \frac{1}{s}\frac{e^{-\tau s}}{\sum_{i=0}^k a_is^{\ga_i}})
     =\int_0^t g(u)du,\quad t>0,
 \end{eqnarray}
 where $g(u)= l(u) - l(u - \tau )$ and $l$ is
 the fundamental solution to (\ref{LFDE}) given by (\ref{l-fund}).
 Assume that $w(s_0)= \sum_{i=0}^{k}a_{i}s_0^{\ga_i}=0$
 for $s_0=u+iv$, $u>0$.
 Then by (\ref{res})
 \begin{equation*}
 l(t)= {\frac{e^{st}}{w^{(1)}(s)}}\left|_{s=s_0}\right. +
    {\sum_{j=1}^k (\frac{e^{st}}{w^{(1)}(s)})}\left|_{s=s_j}\right. +
    l_0(t),\quad t>0.
 \end{equation*}
 Further, note that
 \begin{equation}\label{dod4}
  \frac{e^{st}}{w^{(1)}(s)}\left|_{s=s_0}\right. =
  \frac{e^{ut}[\cos(vt) + i \sin(vt)]}{w^{(1)}(s_0)},\quad t>0,
 \end{equation}
 represents oscillations with increasing amplitudes.
 Inserting (\ref{dod2}) in (\ref{dod1}) we obtain
  \begin{equation}\label{dod3}
   A_{d}= \lim_{t\to 0}\int_0^tg(u)du - \int_0^{\tau }g(u)du \geq 0, \quad \tau,t>0.
 \end{equation}
 It is obvious that due to the presence of the term (\ref{dod4}) in
 $g(t)$ the inequality (\ref{dod3}) could be violated by
 a suitable choice of $\tau $.

\medskip
\noindent

\label{lastpage}

\end{document}